\documentclass[12pt]{article}
\usepackage{graphicx}
\usepackage{amsmath,amssymb,amsthm,amsfonts}
\newtheorem{thm}{Theorem}[section]

\newtheorem{lem}[thm]{Lemma}

\newcommand{\R}{{\mathbb{R}}}
\newcommand{\Z}{{\mathbb{Z}}}

\newcommand{\1}{\partial}
\newcommand{\2}{\overline}
\newcommand{\3}{\varepsilon}
\newcommand{\4}{\widetilde}
\newcommand{\5}{\underline}

\begin{document}
\title{Growth rate and extinction rate of a reaction diffusion equation
with a singular nonlinearity}
\author{Kin Ming Hui\\
Institute of Mathematics, Academia Sinica,\\
Nankang, Taipei, 11529, Taiwan, R. O. C.}
\date{Aug 6, 2008}
\smallbreak \maketitle
\begin{abstract}
We prove the growth rate of global solutions of the equation
$u_t=\Delta u-u^{-\nu}$ in $\R^n\times (0,\infty)$, $u(x,0)=u_0>0$ in $\R^n$,
where $\nu>0$ is a constant. More precisely for any $0<u_0\in C(\R^n)$
satisfying $A_1(1+|x|^2)^{\alpha_1}\le u_0\le A_2(1+|x|^2)^{\alpha_2}$
in $\R^n$ for some constants 
$1/(1+\nu)\le\alpha_1<1$, $\alpha_2\ge\alpha_1$ and $A_2\ge A_1=
(2\alpha_1(1-\3)(n+2\alpha_1-2))^{-1/(1+\nu)}$ where $0<\3<1$ is a constant, 
the global solution $u$ exists and satisfies
$A_1(1+|x|^2+b_1t)^{\alpha_1}\le u(x,t)\le A_2(1+|x|^2+b_2t)^{\alpha_2}$ 
in $\R^n\times (0,\infty)$ where $b_1=2(n+2\alpha_1-2)\3$ and $b_2=2n$ 
if $0<\alpha_2\le 1$ and $b_2=2(n+2\alpha_2-2)$ if $\alpha_2>1$.
When $0<u_0\le A(T_1+|x|^2)^{1/(1+\nu)}$ in $\R^n$ for some constant 
$0<A<((1+\nu)/2n)^{1/(1+\nu)}$, we prove that $u(x,t)
\le A(b(T-t)+|x|^2)^{1/(1+\nu)}$ in $\R^n\times (0,T)$
for some constants $b>0$ and $T>0$. Hence the solution extincts at 
the origin at time $T$. We also find various other conditions for 
the solution to extinct in a finite time and obtain the corresponding
decay rate of the solution near the extinction time. 
\end{abstract}

\vskip 0.2truein


\vskip 0.2truein
\setcounter{equation}{0}
\setcounter{section}{-1}

\section{Introduction}
\setcounter{equation}{0}
\setcounter{thm}{0}

In this paper we will study the growth rate of global solutions and 
behaviour near extinction time of the solution of the following 
Cauchy problem

\begin{equation}
\left\{\begin{aligned}
u_t=&\Delta u-u^{-\nu}\quad\mbox{ in }\R^n\times (0,\infty)\\
u(x,0)=&u_0\qquad\qquad\mbox{ in }\R^n
\end{aligned}\right.
\end{equation}
for any $n\ge 1$ where $\nu>0$ is a constant. Equation (0.1) arises in many 
physical models. When $\nu=2$, (0.1) appears in the modelling of 
Micro-electromechanical systems (MEMS) which 
has many applications such as accelerometers for airbag deployment in cars, 
inkjet printer heads, and the device for the protection of hard disk, etc. 
Interested readers can read the book, Modeling MEMS and NEMS \cite{PB}, 
by J.A.~Pelesko and D.H.~Berstein for the mathematical modeling and various 
applications of MEMS devices. 

Recently there are a lot of research on (0.1) by N.~Ghoussoub, Y.~Guo,
Z.~Pan and M.J.~Ward \cite{GG1}, \cite{GG2}, \cite{GPW}, 
N.I.~Kavallaris, T.~Miyasita and T.~Suzuki \cite{KMS}, 
F.~Lin and Y.~Yang \cite{LY}, L.~Ma, Z.~Guo and J.~Wei \cite{GW1}, 
\cite{GW2}, \cite{MW}, etc. on the various properties of the 
solutions of the equation. Note that the stationary solution
of (0.1) is 
\begin{equation}
\Delta u=u^{-\nu}\quad\mbox{ in }\R^n
\end{equation}
is studied extensively in \cite{GW1}. An equation similar to (0.2) 
arising from the motion of thin films of viscous fluid is also studied 
by H.~Jiang and W.M.~Ni in \cite{JN}.

In \cite{GW1} Z.~Guo and J.~Wei constructed  solutions of (0.1) using a fixed 
point argument. Then by carefully studying the properties of the solutions 
of (0.2) Z.~Guo and J.~Wei \cite{GW1} proved that if $\nu>0$, $n\ge 3$, 
\begin{align*}
u_0\in C_{LB}(\R^n)=&\{\phi\in C(\R^n):\inf_{\R^n}\phi>0\mbox{ and }
\exists\alpha\le 0, A>0\mbox{ and }M>0\mbox{ such that }\\
&\phi (x)\le A|x|^{\alpha}\quad\forall |x|\ge M\}
\end{align*}
and
\begin{equation*}
u_0\ge\gamma\biggl[\frac{2}{\nu+1}\biggl(n-2+\frac{2}{\nu+1}\biggr)
\biggr]^{-\frac{1}{1+\nu}}|x|^{\frac{2}{1+\nu}}
\end{equation*}
for some constant $\gamma>1$,
then (0.1) has a unique global solution
$u$ satisfying
\begin{equation*}
u(x,t)\ge\gamma\biggl[\frac{2}{\nu+1}\biggl(n-2+\frac{2}{\nu+1}\biggr)
\biggr]^{-\frac{1}{1+\nu}}|x|^{\frac{2}{1+\nu}}
\end{equation*}
and
\begin{equation*}
u(x,t)\ge (\nu+1)^{\frac{1}{1+\nu}}(\gamma^{\nu+1}-1)^{\frac{1}{1+\nu}}
t^{\frac{1}{1+\nu}}.
\end{equation*}
They also used a contradiction argument to prove the finite extinction 
property of the solution of (0.1) when $n\ge 3$ and the initial value 
$u_0$ is bounded above by the supersolution of (0.2).

In this paper by approximating the solution of (0.1) by solutions of (0.1)
in bounded domains we prove that for any $n\ge 1$ if the initial value 
$u_0$ satisfies
\begin{equation}
A_1(1+|x|^2)^{\alpha_1}\le u_0\le A_2(1+|x|^2)^{\alpha_2}\quad\mbox{ in }
\R^n
\end{equation}
for some constants $1/(1+\nu)\le\alpha_1<1$, $\alpha_2\ge\alpha_2$, $\nu>0$, 
$A_2\ge A_1$ where
\begin{equation}
A_1=[2\alpha_1(1-\3)(n+2\alpha_1-2)]^{-\frac{1}{1+\nu}}
\end{equation}
for some constant $0<\3<1$, then 
\begin{equation}
A_1(1+|x|^2+b_1t)^{\alpha_1}\le u(x,t)\le A_2(1+|x|^2+b_2t)^{\alpha_2}
\end{equation}
in $\R^n\times (0,\infty)$ where
\begin{equation}
b_1=2(n+2\alpha_1-2)\3
\end{equation}
and
\begin{equation}
b_2=\left\{\begin{aligned}
&2n\qquad\qquad\qquad\mbox{ if }1/(1+\nu)\le\alpha_2\le 1\\
&2(n+2\alpha_2-2)\quad\mbox{ if }\alpha_2>1.
\end{aligned}\right.
\end{equation}
Finite time extinction of solutions of (0.1) when the initial value
$u_0$ is bounded above by the supersolution of (0.2) has been proved
in \cite{GW1}. However there is no estimate for the extinction rate
and extinction time of the solution in \cite{GW1}. In this paper
we will prove the extinction rate and find an explicit upper bound 
for the extinction time of the solutions of (0.1) when
$u_0\in C(\R^n)$, $n\ge 1$, satisfies either
\begin{equation}
0<u_0\le A(T_1+|x|^2)^{\frac{1}{1+\nu}}\quad\forall x\in\R^n
\end{equation}
for some constants 
\begin{equation}
T_1>0\quad\mbox{ and }\quad 0<A<\biggl (\frac{1+\nu}{2n}
\biggr )^{\frac{1}{1+\nu}},
\end{equation}
or
\begin{equation*}
0<u_0\in C(\R^n)\cap L^{\infty}(\R^n).
\end{equation*}
We also find various other conditions on the initial value for
the solution of (0.1) to extinct in a finite time and prove the 
corresponding extinction rate.

The plan of the paper is as follows. In section 1 we will construct
the solution of (0.1) by approximating the solution of (0.1) by solutions
of (0.1) in bounded domains. We will construct explicit supersolutions and
subsolutions of (0.1) and use them to prove that the solutions in bounded
domain have the bounds we want. Then by an approximation argument the 
global solution will have the same upper and lower bounds. 
In section 2 we will prove the 
extinction rate and extinction time of solutions of 
(0.1) under various conditions on the initial value.

We start will some definitions. For any $R>0$, $T>0$, let 
$B_R=\{x\in\R^n:|x|<R\}$, $Q_R^T=B_R\times (0,T)$ and $Q_R=B_R\times 
(0,\infty)$.
For any $f\in C(\1 B_R\times (0,T))\cap L^{\infty}(B_R\times (0,T))$
we say that $0<u\in C([0,T];L^1(B_R))\cap L^{\infty}(Q_R^T)$ is a 
continuous weak solution (subsolution, supersolution respectively) of
\begin{equation}
\left\{\begin{aligned}
u_t=&\Delta u-u^{-\nu}\quad\mbox{ in }Q_R^T\\
u=&f\qquad\qquad\quad\mbox{on }\1B_R\times (0,T)\\
u(x,0)=&u_0\qquad\qquad\mbox{ in }B_R
\end{aligned}\right.
\end{equation}
if $u$ satisfies the integral identity
\begin{equation*}
\iint_{B_R\times [t_1,t_2]}(u\eta_t+u\Delta\eta-u^{-\nu}\eta)\,dx\,dt
=\int_{t_1}^{t_2}\int_{\1 B_R}f\frac{\1\eta}{\1 N}\,d\sigma\,dt
+\left.\int_{B_R}u\eta\,dx\right|_{t=t_1}^{t=t_2}
\end{equation*}
($\ge, \le$ respectively) for any $0<t_1<t_2<T$, 
$\eta\in C^{\infty}(\2{Q_R^T})$ such that $\eta=0$ on $\1 B_R\times [0,T]$
where $\1/\1 N$ is the exterior normal derivative on $\1 B_R\times (0,T)$.

We say that $u$ is a solution (subsolution, supersolution respectively)
of (0.10) if $0<u\in C^{2,1}(Q_R^T)\cap C(\2{B}_R\times (0,T))$ satisfies
($\le, \ge$ respectively) 
\begin{equation}
u_t=\Delta u-u^{-\nu}
\end{equation}
in $Q_R^T$ in the classical sense with $u=f$ ($\le f, \ge f$ respectively) on 
$\1 B_R\times (0,T)$ and 
\begin{equation*}
\lim_{t\to 0}\int_{B_R}|u-u_0|\,dx=0.
\end{equation*}
We say that $u$ is a solution (subsolution, supersolution respectively)
of (0.1) in $\R^n\times (0,T)$ if $u\in C^{2,1}(\R^n\times (0,T))$ 
satisfies (0.11) ($\le, \ge$ respectively) in $\R^n\times (0,T)$ in the 
classical sense, $u(\cdot,t)$ converges to $u_0$ in
$L_{loc}^1(\R^n)$ as $t\to 0$, and for any $0<\delta<T$ there exist 
constants $C_1=C_1(\delta,T)>0$, $C_2=C_2(\delta,T)>0$, such that
\begin{equation*}
C_1(1+|x|^2)^{-\frac{1}{1+\nu}}\le u_2(x,t)\le C_2e^{C_2|x|^2} 
\quad\mbox{ in }\R^n\times (0,T-\delta).
\end{equation*}
We say that $u$ is a solution (subsolution, supersolution respectively)
of (0.1) in $\R^n\times (0,\infty)$ if $u\in C^{2,1}(\R^n\times (0,\infty))$
is a solution of (0.1) in $\R^n\times (0,T)$ 
(subsolution, supersolution respectively) for any $T>0$.

Let $G_R(x,y,t)$, $x,y\in B_R$, $t>0$, be the Dirichlet Green function of 
the heat equation in $Q_R$. That is for any $y\in B_R$,
$$\left\{\aligned 
&\1_tG_R=\Delta_xG_R\quad\text{ in }Q_R\\
&G_R(x,y,t)=0\quad\forall x\in\1 B_R, t>0\\
&\lim_{t\to 0}G_R(x,y,t)=\delta_y
\endaligned\right.
$$
where $\delta_y$ is the delta mass at $y$. By the maximum principle,
$$
0\le G_R(x,y,t)\le\frac{1}{(4\pi)^{\frac{n}{2}}}e^{-|x-y|^2/4t}
\quad\forall x,y\in B_R,t>0.
$$
For any $K\subset R^n\times (0,\infty)$, $0<\beta<1$, let
$$
C^{2,1}(K)=\{f: f,f_t, f_{x_i}, f_{x_i,x_j}\in C(K)\quad\forall 
i,j=1,2,\dots,n\}
$$
and let $C^{2+\beta,1+(\beta/2)}(K)$ denote the class of
all functions $f\in C^{2,1}(K)$ such that
$$\left\{\aligned
&|f_{x_i,x_j}(x_1',t_1')- f_{x_i,x_j}(x_2',t_2')|
\le C(|x_1'-x_2'|^{\beta}+|t_2'-t_1'|^{\beta/2})\}\quad\forall
(x_1',t_1'), (x_2',t_2')\in K\\
&|f_t(x_1',t_1')- f_t(x_2',t_2')|
\le C(|x_1'-x_2'|^{\beta}+|t_2'-t_1'|^{\beta/2})\}\qquad\qquad\forall
(x_1',t_1'), (x_2',t_2')\in K\endaligned\right. 
$$
holds for some constant $C>0$ and any $i,j=1,2,\cdots ,n$.

\section{Existence and growth rate of global solutions}
\setcounter{equation}{0}
\setcounter{thm}{0}

In this section we will construct explicit supersolutions and
subsolutions of (0.1). We will also give a different proof of the
existence of global solutions of (0.1) and prove the growth rate estimates
of the global solutions of (0.1). 

We first observe that by an argument similar to the proof of Theorem 1.1 
of \cite{H2} (cf. Lemma 2.3 of \cite{DK}) we have the following theorem.

\begin{lem}
Let $u_{0,1}, u_{0,2}\in L^1(B_R)$ be such that $0\le u_{0,1}
\le u_{0,2}$ a.e. on $B_R$. Let $f_1,f_2\in C(\1 B_R\times (0,T))\cap 
L^{\infty}(\1 B_R\times (0,T))$ be 
such that $0<f_1\le f_2$ on $\1 B_R\times (0,T)$. Let $u_1$, $u_2$, be 
continuous weak subsolution and supersolution of (0.10) with initial 
value $u_0=u_{0,1}, u_{0,2}$, and $f=f_1,f_2$ respectively. Suppose
there exists a constant $\delta>0$ such that $u_1\ge\delta$ and 
$u_2\ge\delta$ on $Q_R^T$. Then
\begin{equation*}
u_1(x,t)\le u_2(x,t)\quad\forall x\in\2{B}_R,0\le t<T.
\end{equation*}
\end{lem}

\begin{lem}
Let $\nu>0$, $u_0\in C(\2{B}_R)$ and $f\in C(\1 B_R\times (0,T_1))\cap 
L^{\infty}(\1 B_R\times (0,T))$. 
Suppose $\delta=\min(\min_{\2{B}_R}u_0,\inf_{\1 B_R\times (0,T_1)}f)>0$. Then 
there exists a constant $0<T<T_1$ such that (0.10) has a unique solution $u$ 
which satisfies
\begin{equation}
\frac{\delta}{2}\le u\le w_R\quad\forall x\in B_R, 
0<t<T
\end{equation}
and
\begin{align}
u(x,t)=&\int_{B_R}G_R(x,y,t)u_0(y)\,dy
-\int_0^t\int_{\1 B_R}\frac{\1 G_R}{\1 N_y}
(x,y,t-s)f(y,s)\,d\sigma(y)\,ds\nonumber\\
&\qquad -\int_0^t\int_{B_R}G_R(x,y,t-s)u(y,s)^{-\nu}\,dy\,ds
\quad\forall x\in B_R, 0<t<T
\end{align}
where $\1/\1 N_y$ is the exterior normal derivative on $\1 B_R\times (0,T)$
and $w_R$ is the solution of problem
\begin{equation}
\left\{\begin{aligned}
w_t=&\Delta w\quad\mbox{ in }Q_R\\
w=&f\qquad\mbox{on }\1 B_R\times (0,\infty)\\
w(x,0)=&u_0\qquad\mbox{ in }B_R
\end{aligned}\right.
\end{equation}
\end{lem}
\begin{proof}
Since uniqueness of solution of (0.10) follows directly by Lemma 1.1,
we only need to prove existence of solution of (0.10).
Let $u_1(x,t)\equiv\delta$.
For any $k\ge 2$, let
\begin{align}
u_k(x,t)
=&\int_{B_R}G_R(x,y,t)u_0(y)\,dy-\int_0^t\int_{\1 B_R}\frac{\1 G_R}{\1 N_y}
(x,y,t-s)f(y,s)\,d\sigma(y)\,ds\nonumber\\
&\qquad -\int_0^t\int_{B_R}G_R(x,y,t-s)u_{k-1}(y,s)^{-\nu}\,dy\,ds.
\end{align}
Note that by standard parabolic theory \cite{LSU} the solution of (1.3) is 
given by 
\begin{equation}
w_R(x,t)=\int_{B_R}G_R(x,y,t)u_0(y)\,dy
-\int_0^t\int_{\1 B_R}\frac{\1 G_R}{\1 N_y}(x,y,t-s)f(y,s)\,d\sigma(y)\,ds.
\end{equation}
Then by (1.4) and (1.5),
\begin{equation}
u_k(x,t)\le w_R(x,t)\quad\forall |x|\le R,0<t<T_1,k\ge 2.
\end{equation}
By the maximum principle,
\begin{equation*}
w_R\ge\delta\quad\mbox{ in }Q_R^{T_1}.
\end{equation*} 
Let 
\begin{equation}
T=\min (T_1/2,(\delta/2)^{1+\nu}).
\end{equation} 
Then by (1.4) $\forall
(x,t)\in\2{Q_R^T}$,
\begin{equation*}
u_2(x,t)\ge\delta-(\delta/2)^{-\nu}\int_0^t\int_{B_R}G(x,y,t-s)\,dy\,ds
\ge\delta-(\delta/2)^{-\nu}t\ge\delta/2.
\end{equation*}
Suppose $u_k(x,t)\ge\delta/2$ for any $(x,t)\in\2{Q_R^T}$ for some
$k\ge 2$. Then by (1.4) for any $(x,t)\in\2{Q_R^T}$, 
\begin{equation*}
u_{k+1}(x,t)\ge\delta-(\delta/2)^{-\nu}\int_0^t\int_{B_R}G(x,y,t-s)\,dy\,ds
\ge\delta-(\delta/2)^{-\nu}t\ge\delta/2.
\end{equation*}
Hence by induction 
\begin{equation}
u_k(x,t)\ge\delta/2\quad\forall (x,t)\in\2{Q_R^T},k\in\Z^+.
\end{equation}
We now divide the proof into two cases.

\noindent $\underline{\text{\bf Case 1}}$: $u_0\in C^{\infty}(\2{B}_R)$
and $f\in C^{\infty}(\1 B_R\times [0,T_1])$ with $u(x,0)=f(x,0)$ on $\1 B_R$.

By (1.4), (1.6), (1.8) and the parabolic Schauder estimates \cite{LSU},
the sequence $\{u_k\}_{k=1}^{\infty}$ are uniformly Holder continuous 
on $\2{Q_R^T}$. Then by the 
parabolic Schauder estimates (\cite{F},\cite{LSU}) the sequence 
$\{u_k\}_{k=1}^{\infty}$ are uniformly bounded in
$C^{2+\beta,1+(\beta/2)}(K)$ for any compact subset 
$K\subset\2{B}_R\times (0,T]$
where $0<\beta<1$ is some constant. By the Ascoli theorem and a 
diagonalization argument $\{u_k\}_{k=1}^{\infty}$ has a subsequence
which we may assume without loss of generality to be the sequence itself
which converges uniformly in $C^{2+\beta,1+(\beta/2)}(K)$ 
to some function $u\in C(\2{Q_R^T})\cap C^{2+\beta,1+(\beta/2)}(K)$ for any
compact subset $K\subset\2{B}_R\times (0,T]$ as $k\to\infty$. Letting
$k\to\infty$ in (1.4), (1.6) and (1.8), we get (1.1) and (1.2). By (1.4) 
$u_k$ satisfies
\begin{equation}
\left\{\begin{aligned}
u_{k,t}=&\Delta u_k-u_{k-1}^{-\nu}\quad\mbox{ in }Q_R^T\\
u_k=&f\qquad\qquad\quad\mbox{ on }\1 B_R\times (0,T)\\
u_k(x,0)=&u_0\qquad\qquad\quad\mbox{ on }B_R
\end{aligned}\right.
\end{equation}
Letting $k\to\infty$ in (1.9), $u$ satisfies (0.10).

\noindent $\underline{\text{\bf Case 2}}$: $u_0\in C(\2{B}_R)$
and $f\in C(\1 B_R\times [0,T_1])$.

Let $T$ be given by (1.7) and let $\4{u}_1\equiv\delta$. For any $k\ge 2$, let
\begin{align}
\4{u}_k(x,t)
=&\delta\int_{B_R}G_R(x,y,t)\,dy
-\delta\int_0^t\int_{\1 B_R}\frac{\1 G_R}{\1 N_y}
(x,y,t-s)\,d\sigma(y)\,ds\nonumber\\
&\qquad -\int_0^t\int_{B_R}G_R(x,y,t-s)\4{u}_{k-1}(y,s)^{-\nu}\,dy\,ds
\nonumber\\
=&\delta -\int_0^t\int_{B_R}G_R(x,y,t-s)\4{u}_{k-1}(y,s)^{-\nu}\,dy\,ds.
\end{align}
Then
\begin{equation}
\4{u}_k(x,t)\le\delta\quad\mbox{ in }Q_R^{T_1}\quad\forall k\in\Z^+.
\end{equation}
By the same argument as before
\begin{equation}
\4{u}_k(x,t)\ge\delta/2\quad\mbox{ in }Q_R^T\quad\forall k\in\Z^+.
\end{equation}
Hence by case 1 $\4{u}_k$ has a subsequence which we may assume without loss
of generality to be the sequence itself that converges uniformly in 
$C^{2+\beta,1+(\beta/2)}(K)$ to some function $\4{u}\in C(\2{Q_R^T})
\cap C^{2+\beta,1+(\beta/2)}(K)$ for any
compact subset $K\subset\2{B}_R\times (0,T]$ as $k\to\infty$. Moreover $\4{u}$
is a solution of (0.10) with $u_0$ and $f$ being replaced by $\delta$ and 
satisfies
\begin{equation}
\delta/2\le\4{u}(x,t)\le\delta\quad\mbox{ in }Q_R^T.
\end{equation}
Now since $u_1\equiv\4{u}_1$ in $Q_R^T$, by (1.4) and (1.10), 
$\4{u}_2\le u_2$ in $Q_R^T$. Suppose 
\begin{equation}
\4{u}_k\le u_k\quad\mbox{ in }Q_R^T.
\end{equation}
for some $k\ge 2$. Then by (1.4), (1.10) and (1.14) we get that (1.14) 
holds with $k$
replaced by $k+1$. Hence by induction (1.14) holds for all $k\in\Z^+$.
Since $\4{u}_k$ converge uniformly to $\4{u}$ on $\2{Q_R^T}$ as $k\to\infty$,
by (1.6) and (1.14) and the parabolic Schauder estimates 
(\cite{F},\cite{LSU}) the sequence 
$\{u_k\}_{k=1}^{\infty}$ are uniformly bounded in
$C^{2+\beta,1+(\beta/2)}(K)$ for any compact subset 
$K\subset B_R\times (0,T]$
where $0<\beta<1$ is some constant. By the Ascoli theorem and a 
diagonalization argument $\{u_k\}_{k=1}^{\infty}$ has a subsequence
which we may assume without loss of generality to be the sequence itself
which converges uniformly in $C^{2+\beta,1+(\beta/2)}(K)$ 
to some function $u\in C(Q_R^T)\cap C^{2+\beta,1+(\beta/2)}(K)$ for any
compact subset $K\subset B_R\times (0,T]$ as $k\to\infty$. 

By (1.4) $u_k$ satisfies (1.9). Letting 
$k\to\infty$ in (1.4) and (1.9) $u$ satisfies (1.2) and (0.11) in $Q_R^T$. 
Then by (1.2) $u\in C(\2{B}_R\times (0,T))$. Since the last two terms of
(1.2) vanish as $t\to 0$ and the first term on the right hand side is the
solution of (1.3) with $f=0$ which converges to $u_0$ in $L^1(B_R)$ as 
$t\to 0$, $u$ converges to $u_0$ in $L^1(B_R)$ as $t\to 0$. Hence $u$ is the
solution of (0.10). Letting $k\to\infty$ in (1.6) and (1.14), by (1.13) we
get (1.1) and the lemma follows.
\end{proof}

\begin{thm}
Let $\nu>0$, $\alpha_2\ge\alpha_1\ge 1/(1+\nu)$
and let 
\begin{equation}
A_2\ge A_1=\left\{\begin{aligned}
&[2\alpha_1(1-\3)(n+2\alpha_1-2)]^{-\frac{1}{1+\nu}}\quad\mbox{ if }
1/(1+\nu)\le\alpha_1\le 1\\
&[2\alpha_1(1-\3)n]^{-\frac{1}{1+\nu}}\qquad\qquad\qquad\mbox{ if }\alpha_1>1
\end{aligned}\right.
\end{equation}
for some constant $0<\3<1$. Suppose $u_0$ satisfies 
(0.3). Let $f\in C(\1 B_R\times [0,\infty))$ be such that
$$
A_1(1+|x|^2+b_1t)^{\alpha_1}\le f\le A_2(1+|x|^2+b_2t)^{\alpha_2}
\quad\mbox{ on }\1 B_R\times [0,\infty)
$$
where 
\begin{equation}
b_1=\left\{\begin{aligned}
&2(n+2\alpha_1-2)\3\quad\mbox{ if }1/(1+\nu)\le\alpha_1\le 1\\
&2\3 n\qquad\qquad\qquad\mbox{ if }\alpha_1>1
\end{aligned}\right.
\end{equation}
and $b_2$ is given by (0.7). Then there exists a unique solution $u$ of 
(0.10) in $Q_R$ which satisfies (1.2) and (0.5) in $Q_R$. 
\end{thm}
\begin{proof}
Let $\psi_i=A_i(1+|x|^2+b_it)^{\alpha_i}$ for $i=1,2$. By direct 
computation, $\forall i=1,2$,
\begin{equation}
\Delta\psi_i=2\alpha_iA_i\biggl\{n+2\alpha_i-2+2(1-\alpha_i)
\frac{1+b_it}{1+|x|^2+b_it}\biggr\}(1+|x|^2+b_it)^{\alpha_i-1}
\end{equation}
Since
\begin{equation*}
2(1-\alpha_1)\frac{1+b_1t}{1+|x|^2+b_1t}\ge\left\{\begin{aligned}
&0\qquad\qquad\quad\mbox{if }1/(1+\nu)\le\alpha_1\le 1\\
&2(1-\alpha_1)\quad\,\mbox{ if }\alpha_1>1,
\end{aligned}\right.
\end{equation*}
by (1.17),
\begin{equation*}
\Delta\psi_1\ge\left\{\begin{aligned}
&2\alpha_1A_1(n+2\alpha_1-2)(1+|x|^2+b_1t)^{\alpha_1-1}
\quad\mbox{if }1/(1+\nu)\le\alpha_1\le 1\\
&2\alpha_1nA_1(1+|x|^2+b_1t)^{\alpha_1-1}\qquad\qquad
\qquad\mbox{ if }\alpha_1>1.
\end{aligned}\right.
\end{equation*}
Hence for $1/(1+\nu)\le\alpha_1\le 1$,
\begin{align}
&\Delta\psi_1-\psi_1^{-\nu}-\psi_{1,t}\nonumber\\
\ge&2\alpha_1A_1(n+2\alpha_1-2)(1+|x|^2+b_1t)^{\alpha_1-1}
-A_1^{-\nu}(1+|x|^2+b_1t)^{-\alpha_1\nu}\nonumber\\
&\qquad-\alpha_1b_1A_1(1+|x|^2+b_1t)^{\alpha_1-1}\nonumber\\
\ge&A_1^{-\nu}(1+|x|^2+b_1t)^{-\alpha_1\nu}\{\alpha_1(2(n+2\alpha_1-2)-b_1)
A_1^{1+\nu}(1+|x|^2+b_1t)^{\alpha_1(1+\nu)-1}-1\}\nonumber\\
\ge&A_1^{-\nu}(1+|x|^2+b_1t)^{-\alpha_1\nu}\{\alpha_1(2(n+2\alpha_1-2)-b_1)
A_1^{1+\nu}-1\}
\end{align}
and for $\alpha_1>1$,
\begin{align}
&\Delta\psi_1-\psi_1^{-\nu}-\psi_{1,t}\nonumber\\
\ge&2\alpha_1nA_1(1+|x|^2+b_1t)^{\alpha_1-1}
-A_1^{-\nu}(1+|x|^2+b_1t)^{-\alpha_1\nu}\nonumber\\
&\qquad-\alpha_1b_1A_1(1+|x|^2+b_1t)^{\alpha_1-1}\nonumber\\
\ge&A_1^{-\nu}(1+|x|^2+b_1t)^{-\alpha_1\nu}\{\alpha_1(2n-b_1)
A_1^{1+\nu}(1+|x|^2+b_1t)^{\alpha_1(1+\nu)-1}-1\}\nonumber\\
\ge&A_1^{-\nu}(1+|x|^2+b_1t)^{-\alpha_1\nu}\{\alpha_1(2n-b_1)
A_1^{1+\nu}-1\}.
\end{align}
By (1.15) and (1.16) the right hand side of (1.18) and (1.19) is $\ge 0$.
Hence
\begin{equation*}
\Delta\psi_1-\psi_1^{-\nu}\ge\psi_{1,t}\quad\mbox{ in }Q_R.
\end{equation*}
Since
\begin{equation*}
2(1-\alpha_2)\frac{1+b_2t}{1+|x|^2+b_2t}\le\left\{\begin{aligned}
&2(1-\alpha_2)\quad\mbox{if }1/(1+\nu)\le\alpha_2\le 1\\
&0\qquad\qquad\,\mbox{ if }\alpha_2>1,
\end{aligned}\right.
\end{equation*}
by (1.17),
\begin{equation*}
\Delta\psi_2\le\left\{\begin{aligned}
&2\alpha_2nA_2(1+|x|^2+b_2t)^{\alpha_2-1}\qquad\qquad\qquad
\,\,\mbox{ if }1/(1+\nu)\le\alpha_2\le 1\\
&2\alpha_2A_2(n+2\alpha_1-2)(1+|x|^2+b_2t)^{\alpha_2-1}
\quad\mbox{ if }\alpha_2>1.
\end{aligned}\right.
\end{equation*}
By (0.7) for $1/(1+\nu)\le\alpha_2\le 1$,
\begin{align}
&\Delta\psi_2-\psi_2^{-\nu}-\psi_{2,t}\nonumber\\
\le&2\alpha_2nA_2(1+|x|^2+b_2t)^{\alpha_2-1}
-A_2^{-\nu}(1+|x|^2+b_2t)^{-\alpha_2\nu}\nonumber\\
&\qquad-\alpha_2b_2A_2(1+|x|^2+b_2t)^{\alpha_2-1}\nonumber\\
\le&A_2^{-\nu}(1+|x|^2+b_2t)^{-\alpha_2\nu}\{\alpha_2(2n-b_2)
A_2^{1+\nu}(1+|x|^2+b_2t)^{\alpha_2(1+\nu)-1}-1\}\nonumber\\
\le&0\quad\mbox{ in }Q_R\nonumber
\end{align}
and for $\alpha_1>1$,
\begin{align}
&\Delta\psi_2-\psi_2^{-\nu}-\psi_{2,t}\nonumber\\
\le&2\alpha_2A_2(n+2\alpha_2-2)(1+|x|^2+b_2t)^{\alpha_2-1}
-A_2^{-\nu}(1+|x|^2+b_2t)^{-\alpha_2\nu}\nonumber\\
&\qquad-\alpha_2b_2A_2(1+|x|^2+b_2t)^{\alpha_2-1}\nonumber\\
\le&A_2^{-\nu}(1+|x|^2+b_2t)^{-\alpha_2\nu}\{\alpha_2(2(n+2\alpha_2-2)-b_2)
A_2^{1+\nu}(1+|x|^2+b_2t)^{\alpha_2(1+\nu)-1}-1\}\nonumber\\
\le&0\quad\mbox{ in }Q_R\nonumber.
\end{align}
By Lemma 1.2 there exists a constant $T>0$ such that there exists a 
unique solution $u$ of (0.10) in $Q_R^T$ 
which satisfies (1.1) and (1.2) in $Q_R^T$ with $\delta=A_1$. 
Let $T_1\ge T$ be the maximal existence time of
a unique solution of $u$ of (0.10) in $Q_R^{T_1}$. Suppose $T_1<\infty$. 
For any $0<\delta<T_1/2$ since $u\in C(\2{B}_R\times [T_1/2,T_1-\delta])$ is
a classical solution of (0.11), $\min_{\2{B}_R\times [T_1/2,T_1-\delta]}u>0$.
Hence by (1.2) $\inf_{\2{Q_R^{T_1-\delta}}}u>0$. 
Since $\psi_1$ and $\psi_2$ are subsolution and supersolution of (0.10), 
by Lemma 1.1 (0.5) holds in 
$Q_R^{T_1-\delta}$ for any $0<\delta<T_1/2$. Hence (0.5) holds in 
$Q_R^{T_1}$. 

Then by (0.5) and the parabolic Schauder estimates \cite{LSU}
$u$ can be extended to a continuous function on $\2{B}_R\times [T_1/2,T_1]$.
Then by Lemma 1.2 there exists a constant $T_3>0$ such that there exists
a solution $\4{u}(x,t)$ of (0.10) in $Q_R^{T_3}$ with $u_0=u(x,T_1)$
and $f$ being replaced by $f(x,T_1+t)$. We extend $u$
to a function on $\2{B}_R\times (0,T_1+T_3)$ by setting $u(x,t)=\4{u}
(x,T_1+t)$ for any $|x|\le R$, $T_1\le t<T_1+T_3$. Then $u$ is a solution
of (0.10) in $Q_R^{T_1+T_3}$. This contradicts the maximality of $T_1$. Hence
$T_1=\infty$. By the previous argument $u$ satisfies (0.5) in $Q_R$.
Let $v$ be given by the right hand side of (1.2). Then $v$ satisfies
\begin{equation*}
\left\{\begin{aligned}
v_t=&\Delta v-u^{-\nu}\quad\mbox{ in }Q_R\\
v=&f\qquad\qquad\,\,\mbox{ on }\1 B_R\times (0,\infty)\\
v(x,0)=&u_0\qquad\qquad\,\mbox{ in }B_R.
\end{aligned}\right.
\end{equation*}
Hence the function $w=u-v$ satisfies
\begin{equation*}
\left\{\begin{aligned}
w_t=&\Delta w\quad\mbox{ in }Q_R\\
w=&0\qquad\mbox{ on }\1 B_R\times (0,\infty)\\
w(x,0)=&0\qquad\mbox{ in }B_R.
\end{aligned}\right.
\end{equation*}
By the maximum principle $w\equiv 0$ in $Q_R$. Hence $u=v$ in $Q_R$. 
Thus $u$ satisfies (1.2) and the theorem follows.
\end{proof}

We next recall a comparison result of \cite{W}.

\begin{lem}(Lemma 1.3 of \cite{W})
Suppose $\2{u}$ and $\5{u}$ are supersolution and subsolution of
$$\left\{\begin{aligned}
&u_t=\Delta u+f(u)\quad\mbox{ in }\R^n\times (0,T)\\
&u(x,0)=u_0\qquad\quad\mbox{in }\R^n
\end{aligned}\right.
$$
for some function $u_0\in C(\R^n)$ and $f\in C(\R)$ such that 
$\2{u}-\5{u}\ge -Be^{\beta |x|^2}$
in $\R^n\times (0,T)$ for some constants $B>0$ and $\beta>0$. Suppose
$f(\2{u})-f(\5{u})\ge g(x,t)(\2{u}-\5{u})$ for some function 
$g\in C_{loc}^{\alpha,\alpha/2}(\R^n\times (0,T))$, $0<\alpha<1$, and
$g(x,t)\le C(1+|x|^2)$ on $\R^n\times (0,T)$ for some constant $C>0$.
Then $\2{u}\ge\5{u}$ on $\R^n\times (0,T)$.
\end{lem}

\begin{thm}
Let $\nu>0$, $1/(1+\nu)\le\alpha_1\le\alpha_2$, $\alpha_1<1$ and 
$A_2\ge A_1$ where $A_1$ is given by (0.4) for some constant 
$0<\3<1$. Suppose $u_0$ satisfies (0.3). Then there exists a unique 
solution $u$ of (0.1) which satisfies 
\begin{equation}
u(x,t)=\int_{\R^n}\frac{1}{(4\pi t)^{\frac{n}{2}}}e^{-\frac{|x-y|^2}{4t}}
u_0(y)\,dy-\int_0^t\int_{\R^n}\frac{1}{(4\pi (t-s))^{\frac{n}{2}}}
e^{-\frac{|x-y|^2}{4(t-s)}}u(y,s)^{-\nu}\,dy\,ds
\end{equation}
and (0.5) in $\R^n\times (0,\infty)$ where $b_1$ and $b_2$ are given by 
(0.6) and (0.7).
\end{thm}
\begin{proof}
By Theorem 1.3 for any $k\in\Z^+$ there exists a unique solution $u_k$
of (0.10) with $f=A_1(1+|x|^2+b_1t)^{\alpha_1}$ in $Q_k$ which satisfies
(0.5) and 
\begin{align}
u_k(x,t)
=&\int_{B_R}G_k(x,y,t)u_0(y)\,dy-\int_0^t\int_{B_k}G_k(x,y,t-s)
u_k(y,s)^{-\nu}\,dy\,ds\nonumber\\
&\qquad-A_1\int_0^t\int_{\1 B_k}\frac{\1 G_k}{\1 N_y}
(x,y,t-s)(1+|y|^2+b_1s)^{\alpha_1}\,d\sigma(y)\,ds.
\end{align}
in $Q_k$. By (0.5) and the parabolic Schauder
estimates \cite{LSU} the sequence $\{u_k\}_{k=1}^{\infty}$ is uniformly
bounded in $C^{2+\beta,1+(\beta/2)}(K)$ for any compact subset 
$K\subset\R^n\times (0,\infty)$ where $0<\beta<1$ is some constant. 
By the Ascoli theorem and a diagonalization argument 
$\{u_k\}_{k=1}^{\infty}$ has a subsequence which we may assume without 
loss of generality to be the sequence itself which converges uniformly in 
$C^{2+\beta,1+(\beta/2)}(K)$ to some function $u\in C^{2,1}
(\R^n\times (0,\infty))$ as $k\to\infty$. Then $u$ satisfies (0.11) in
$\R^n\times (0,\infty)$. 

Putting $u=u_k$ in (0.5) and letting $k\to\infty$ we get that $u$ 
satisfies (0.5) in $\R^n\times (0,\infty)$. Since $G_k(x,y,t)$ increases
monotonically to $(4\pi t)^{-n/2}e^{-|x-y|^2/4t}$ as $k\to\infty$, the 
first two terms on the right hand side of (1.21) converges to 
$$
\int_{\R^n}\frac{1}{(4\pi t)^{\frac{n}{2}}}e^{-\frac{|x-y|^2}{4t}}
u_0(y)\,dy-\int_0^t\int_{\R^n}\frac{1}{(4\pi (t-s))^{\frac{n}{2}}}
e^{-\frac{|x-y|^2}{4(t-s)}}u(y,s)^{-\nu}\,dy\,ds
$$
as $k\to\infty$. By Lemma 1.3 of \cite{H1} for any $T>0$
there exists constants $C_1>0$, $c_1>0$ such that
\begin{equation}
\biggl|\frac{\1 G_k}{\1 N_y}(x,y,t-s)\biggr |
\le C_1\frac{k-|x|}{(t-s)^{\frac{n+2}{2}}}e^{-c_1\frac{|x-y|^2}{t-s}}\quad
\forall |x|<k, |y|=k, 0<s<t<T,k\ge 1
\end{equation}
Let $k_0>1$. Then by (1.22) for any $|x|\le k_0$, $k\ge 2k_0$, $|y|=k$,
$0<s<t<T$,
\begin{align*}
\biggl|\frac{\1 G_k}{\1 N_y}(x,y,t-s)\biggr |
\le&C_1\frac{k-|x|}{|x-y|^{n+2}}e^{-\frac{c_1|x-y|^2}{2(t-s)}}\biggl[
\biggl(\frac{|x-y|^2}{(t-s)}\biggr)^{\frac{n+2}{2}}
e^{-\frac{c_1|x-y|^2}{2(t-s)}}\biggr]\\
\le&C_2\frac{k-|x|}{|x-y|^{n+2}}e^{-\frac{c_1|x-y|^2}{2(t-s)}}\\
\le&C_2\frac{k-|x|}{(k-|x|)^{n+2}}
\end{align*}
for some constant $C_2>0$.
Hence
\begin{align}
\biggl|\int_0^t\int_{|y|=k}\frac{\1 G_k}{\1 N_y}(x,y,t-s)
(1+|y|^2+b_1s)^{\alpha_1}\,d\sigma (y)\,ds\biggr|
\le&C\frac{(k-|x|)k^{n-1}}{(k-|x|)^{n+2}}(1+k^2)^{\alpha_1}\nonumber\\
\to&0\quad\mbox{ as }k\to\infty
\end{align}
for any $|x|\le k_0, 0<t<T$. Since $k_0>1$ and $T>0$ are arbitrary,
(1.23) holds for any $x\in\R^n$ and $t>0$. Hence letting $k\to\infty$ in 
(1.21) we get (1.20).

By (0.5) the second term on the right hand side of (1.20) converges 
uniformly to $0$ as $t\to 0$. Since the first term on the right hand 
side of (1.20) is the solution of the heat equation in $\R^n\times 
(0,\infty)$, this term converges in $L_{loc}^1(\R^n)$ to $u_0$ as $t\to 0$. 
Hence by (1.20) $u(\cdot,t)$ converges in $L_{loc}^1(\R^n)$ to $u_0$ as 
$t\to 0$. Thus $u$ is a solution of (0.1). Suppose there exists another 
solution $v$ of (0.1) which satisfies (0.5) and (1.20) in $\R^n\times 
(0,\infty)$. By the mean value theorem,
$$
-u^{-\nu}+v^{-\nu}=\nu\xi^{-(1+\nu)}(u-v)
$$
for some $\xi=\xi(x,t)$ between $u(x,t)$ and $v(x,t)$. Then by (0.5),
$$
\xi^{-(1+\nu)}\le C(1+|x|^2+b_1t)^{-\alpha_1(1+\nu)}\le C\quad\forall
x\in\R^n, t>0,
$$
for some constant $C>0$. Henc by Lemma 1.4 $u=v$ on $\R^n\times (0,\infty)$.
Thus the solution $u$ of (0.1) is unique.
\end{proof}

By Lemma 1.2, Lemma 1.4 and an argument similar to the proof of 
Theorem 1.5 we obtain the following extension of the local existence 
theorem of (0.1) (cf. Theorem 3.3 of \cite{GW1}) proved in \cite{GW1}.

\begin{thm}
Let $\nu>0$. Suppose $u_0\in C(\R^n)$ satisfies $\delta=\inf_{\R^n}u_0>0$
and there exists constants $C_1>0$ and $C_2>0$ such that 
$$
u_0(x)\le C_1e^{C_2|x|^2}\quad\mbox{ in }\R^n.
$$ 
Then there exists a constant $T>0$ such that
there exists a unique solution $u$ of (0.1) in $\R^n\times (0,t)$
which satisfies (1.20) in $\R^n\times (0,\infty)$. If $u_0\le C_3
(1+|x|^2)^{\alpha}$ in $\R^n$ for some constants $C_3>0$ and 
$\alpha\ge 1/(1+\nu)$,
then $T\ge (\delta/2)^{-\nu-1}$ such that when $T<\infty$ we have
$\lim_{t\nearrow T}\inf_{\R^n}u(\cdot,t)=0$.
\end{thm}

\section{Extinction properties of solution}
\setcounter{equation}{0}

In this section we will establish various conditions on the initial 
value $u_0$ for the solutions of (0.1) to extinct in a finite time. 
We will give upper bound estimates for the extinction time and find the
extinction rate of the solutions of (0.1).

\begin{thm}
Let $0<\nu\le 1$,
$$
0<\beta\le\frac{1}{1+\nu}.
$$
Let $u_0\in C(\R^n)$ satisfies
$$
0<u_0\le A_3T^{\frac{1}{1+\nu}}(1+|x|^2)^{-\beta}
$$
for some constant
\begin{equation}
A_3=\left\{\begin{aligned}
&(1+\nu)^{\frac{1}{1+\nu}}\qquad\qquad\qquad\qquad
\mbox{ if }0<\beta\le\min((n-2)/2,1/(1+\nu))\\
&\biggl (\frac{1+\nu}{1+2\beta(1+\nu)(2\beta+2-n)T}
\biggr )^{\frac{1}{1+\nu}}\quad\mbox{ if }(n-2)/2<\beta\le 1/(1+\nu).
\end{aligned}\right.
\end{equation}
Suppose $u$ is a solution of (0.1). Then 
\begin{equation}
u(x,t)\le A_3(T-t)^{\frac{1}{1+\nu}}(1+|x|^2)^{-\beta}\quad\forall x\in\R^n,
0<t<T.
\end{equation}
\end{thm}
\begin{proof}
Suppose $u$ is a solution of (0.1). Let $\psi_3(x,t)
=A_3(T-t)^{\frac{1}{1+\nu}}(1+|x|^2)^{-\beta}$. Then 
$u_0\le\psi_3(x,0)$ on $\R^n$. By direct computation,
$$
\Delta\psi_3
=A_3(T-t)^{\frac{1}{1+\nu}}[2\beta (2\beta +2-n)(1+|x|^2)^{-\beta-1}
-4\beta (\beta+1)(1+|x|^2)^{-\beta-2}]
$$
and
$$
\psi_{3,t}=-\frac{A_3}{1+\nu}(T-t)^{-\frac{\nu}{1+\nu}}(1+|x|^2)^{-\beta}.
$$
Hence
\begin{align}
&\Delta\psi_3-\psi_3^{-\nu}-\psi_{3,t}\nonumber\\
=&A_3(T-t)^{\frac{1}{1+\nu}}(1+|x|^2)^{-\beta-2}\biggl (
2\beta (2\beta +2-n)(1+|x|^2)-4\beta (\beta+1)\nonumber\\
&\qquad-A_3^{-\nu-1}
\frac{(1+|x|^2)^{2+\beta (1-\nu)}}{T-t}+\frac{1}{1+\nu}\cdot
\frac{(1+|x|^2)^2}{T-t}\biggr )\nonumber\\
\le&A_3(T-t)^{\frac{1}{1+\nu}}(1+|x|^2)^{-\beta-2}\biggl (
2\beta (2\beta +2-n)(1+|x|^2)-4\beta (\beta+1)\nonumber\\
&\qquad+\frac{(1+|x|^2)^2}{T-t}\cdot\biggl (\frac{1}{1+\nu}-A_3^{-\nu-1}
\biggr )\biggr ).
\end{align}
By (2.1) if $0<\beta\le\min((n-2)/2,1/(1+\nu))$, then the right hand side of
(2.3) is $\le 0$. If $(n-2)/2<\beta\le 1/(1+\nu)$, let
$$
\3=\frac{1}{1+2\beta(1+\nu)(2\beta+2-n)T}.
$$
Then by (2.1),
$$
A_3=(\3(1+\nu))^{\frac{1}{1+\nu}}
=\biggl (\frac{1-\3}{2\beta(2\beta+2-n)T}
\biggr )^{\frac{1}{1+\nu}}.
$$
Hence
\begin{equation}
\frac{1}{1+\nu}=A_3^{-\nu-1}\3
\end{equation}
and
\begin{equation}
2\beta (2\beta +2-n)=A_3^{-\nu-1}\frac{1-\3}{T}. 
\end{equation}
By (2.4) and (2.5) when $(n-2)/2<\beta\le 1/(1+\nu)$, the right side
of (2.3) is $\le 0$. Hence $\psi_3$ is a supersolution of (0.1). Then 
by Lemma 1.4,
$$
u\le\psi_3\quad\mbox{ in }\R^n\times (0,T)
$$
and (2.2) follows.
\end{proof}

\begin{thm}
Let $\nu>0$, $0<u_0\in C(\R^n)\cap L^{\infty}(\R^n)$, and
$T=\|u_0\|_{L^{\infty}(\R^n)}^{1+\nu}/(1+\nu)$. Suppose $u$ is a solution 
of (0.1). Then 
\begin{equation}
u(x,t)\le (1+\nu)^{\frac{1}{1+\nu}}(T-t)^{\frac{1}{1+\nu}}\quad\forall
x\in\R^n,0<t<T.
\end{equation}
\end{thm}
\begin{proof}
Let $\psi_4(x,t)=(1+\nu)^{\frac{1}{1+\nu}}(T-t)^{\frac{1}{1+\nu}}$.
Then $u_0\le\psi_4(x,0)$ on $\R^n$ and 
$$
\Delta\psi_4-\psi_4^{-\nu}-\psi_{4,t}=0\quad\mbox{ in }\R^n\times (0,T).
$$
Hence by Lemma 1.4, 
$$
u\le\psi_4\quad\mbox{ in }\R^n\times (0,T)
$$
and (2.6) follows.
\end{proof}

\begin{thm}
Let $\nu>0$ and let $0<u_0\in C(\R^n)$ satisfies 
(0.8) for some constants $T_1>0$ and $A>0$ satisfying (0.9). Let $u$
be a solution of (0.1). Then there exist constants $T>0$ and $b>0$ such that
\begin{equation}
u(x,t)\le A(b(T-t)+|x|^2)^{\frac{1}{1+\nu}}\quad\mbox{ in }\R^n\times (0,T).
\end{equation}
Hence $u$ vanishes at $x=0$ at time $t=T$.
\end{thm}
\begin{proof}
Let 
\begin{equation}
b=\frac{1+\nu}{A^{1+\nu}}-2n
\end{equation} 
and $T=T_1/b$. By (0.9) $b>0$. Let $\psi_5(x,t)
=A(b(T-t)+|x|^2)^{\frac{1}{1+\nu}}$. Suppose $u$ is a solution of (0.1). 
By (0.8) 
\begin{equation}
u_0\le\psi_5(x,0)\quad\mbox{ on }\R^n.
\end{equation} 
By direct computation,
\begin{equation*}
\Delta\psi_5\le\frac{2nA}{1+\nu}(|x|^2+b(T-t))^{-\frac{\nu}{1+\nu}}
\quad\mbox{ in }\R^n\times (0,T).
\end{equation*}
Hence by (2.8),
\begin{align}
&\Delta\psi_5-\psi_5^{-\nu}-\psi_{5,t}\nonumber\\
\le&\frac{2nA}{1+\nu}(|x|^2+b(T-t))^{-\frac{\nu}{1+\nu}}
-A^{-\nu}(|x|^2+b(T-t))^{-\frac{\nu}{1+\nu}}
+\frac{bA}{1+\nu}(|x|^2+b(T-t))^{-\frac{\nu}{1+\nu}}\nonumber\\
\le&0\quad\mbox{ in }\R^n\times (0,T).
\end{align}
By (2.9), (2.10) and Lemma 1.4,
\begin{equation*}
u\le\psi_5\quad\mbox{ on }\R^n\times (0,T)
\end{equation*} 
and (2.8) follows.
\end{proof}

\end{document}